\documentclass[12pt,a4paper]{amsart}

\usepackage{amssymb}
\usepackage{amsthm}
\usepackage{amsmath}
\usepackage{graphicx}
\usepackage{enumerate}
\usepackage{xcolor}
\usepackage{tikz}
\usepackage{graphbox}
\usepackage{orcidlink}
\usepackage{mathrsfs}  
\usepackage{hyperref}
\hypersetup{colorlinks=true,allcolors=blue}
\usepackage[margin=2.5cm]{geometry}

\newtheorem{theorem}{Theorem}[section]
\newtheorem{proposition}[theorem]{Proposition}

\newtheorem{lemma}[theorem]{Lemma}

\theoremstyle{definition}

\newtheorem{definition}[theorem]{Definition}

\newcommand{\II}{\mathbb{I}}
\newcommand{\RR}{\mathbb{R}}
\newcommand{\CC}{\mathscr C}

\newcommand{\A}{\mathcal{A}}
\newcommand{\B}{\mathcal{B}}
\newcommand{\C}{\mathcal{C}}
\newcommand{\D}{\mathcal{D}}
\newcommand{\E}{\mathcal{E}}
\newcommand{\F}{\mathcal{F}}
\newcommand{\PP}{\mathcal{P}}
\newcommand{\T}{\mathcal{T}}

\title[The relation between non-exchangeability and radial asymmetry]{The relation between non-exchangeability and radial asymmetry of copulas}

\author[D. Kokol Bukovšek]{Damjana Kokol Bukovšek \orcidlink{0000-0002-0098-6784} }
\address{University of Ljubljana, School of Economics and Business, and Institute of Mathematics, Physics and Mechanics, Ljubljana, Slovenia}
\email{Damjana.Kokol.Bukovsek@ef.uni-lj.si}

\author[T. Košir]{Tomaž Košir \orcidlink{0000-0002-3661-8094} } 
\address{University of Ljubljana, Faculty of Mathematics and Physics, and Institute of Mathematics, Physics and Mechanics, Ljubljana, Slovenia}
\email{Tomaz.Kosir@fmf.uni-lj.si}

\author[B. Mojškerc]{Blaž Mojškerc \orcidlink{0000-0001-8096-355X} }
\address{University of Ljubljana, School of Economics and Business, and Institute of Mathematics, Physics and Mechanics, Ljubljana, Slovenia}
\email{Blaz.Mojskerc@ef.uni-lj.si}

\author[M. Omladič]{Matjaž Omladič \orcidlink{0000-0001-5383-9203} } 
\address{University of Ljubljana, Faculty of Mathematics and Physics,  Ljubljana, Slovenia}
\email{Matjaz@Omladic.net}


\begin{document}

\begin{abstract}
Non-exchangeability and radial asymmetry have got substantial attention in the literature, but relation between the two has not been studied yet. We propose the question: if a copula has a given degree of radial asymmetry, what is the range of possible values of the measure of non-exhangeability.
We determine the set of all points whose first coordinate is the radial asymmetry of a copula and the second coordinate is its non-exchangeability. We show that this set is a non-convex hexagon.
\end{abstract}

\maketitle

\section{Introduction} \label{sec:intro}

Bivariate copulas capture the dependence structure between two random variables.
We can think of a copula as a function of two variables
\[
    C \colon (u,v) \mapsto C(u,v), \ u,v\in\II,
\]
where $\II=[0,1]$, 
satisfying certain properties, see Section~\ref{sec:prelim}.
The domain of a copula is the unit square $\II^2$, which has symmetry group $D_4$, the dihedral group with $8$ elements.
The symmetries of the unit square give rise to various types of symmetry of copulas. The most natural one is an \emph{exchangeability}, i.e., a copula is exchangeable if $C(u,v)=C(v,u) =\colon C^T(u,v)$ for every $u,v\in \II$. The name comes from the fact that the copula does not change when the variables are exchanged. Sometimes exchangeability of copulas is called just symmetry \cite{KleMes06}.

Many important copulas are exchangeable, such as the independence copula,  the Fr\'echet-Hoeffding bounds, Archimedean copulas \cite{Nels}, etc. 
On the other hand, an important example of non-exchange\-able copulas are Marshall-Olkin copulas, see \cite{MaOl67, Ma93}. Non-exchangeability of various classes of copulas including Marshall copulas was studied in \cite{KoKoMoOm19}.

Another type of symmetry occurs when a copula $C$ is equal to its survival copula $\widehat{C}$, see Section~\ref{sec:prelim}. This type of symmetry is called \emph{radial symmetry}. Survival copulas arise naturally in the shock model theory when we replace life times with survival times.

Two more types of symmetries arise from reflections of copulas with respect to the first or the second variable, defined as
$$ C^{\sigma_1}(u,v)=v-C(1-u,v) \qquad \text{and} \qquad C^{\sigma_2}(u,v)=u-C(u,1-v).$$
Note that $(C^{\sigma_1})^{\sigma_2}=(C^{\sigma_2})^{\sigma_1}=\widehat{C}$, so that the symmetry with respect to both of these two transformations implies radial symmetry.
It turns out that transformations
\[
    C\to C, C^{\sigma_1}, C^{\sigma_2}, \widehat{C}, C^T, (C^{\sigma_1})^T,  (C^{\sigma_2})^T, \ \mbox{and}\ \widehat{C}^T
\]
form a group of transformations (with respect to composition of transformations) on the set of bivariate copulas. This group is isomorphic to the dihedral group $D_4$ (see \cite[Section 2]{FuSc14}). We say that a copula $C$ is \emph{$D_4$-symmetric} if it is left unchanged by all the transformations of this group.

When dealing with real-world data, the degree of various types of asymmetry is of interest. The measure of non-exhangeability of a copula $C$ is
$$\mu_T(C) = \max\big\{|C(u,v)-C^T(u,v)| \colon u,v\in \II\big\},$$
while the measure of its radial asymmetry is
$$\mu_R(C) = \max\big\{|C(u,v)-\widehat{C}(u,v)| \colon u,v\in \II\big\}.$$

Non-exchangeability of copulas was first studied by Klement and Mesiar in \cite{KleMes06} where  the asymmetry function was introduced, and by Nelsen in \cite{Ne07}. There, both non-exchange\-abi\-lity and radial asymmetry, and their asymmetry functions were considered. Later, non-exchangeability in norms other than the $L_{\infty}$-norm was studied  in \cite{DuKlSeUF10,GeNe13}, and the radial asymmetry in \cite{AVDo18,DeDoUF13}. Genest and Ne\v slehov\'a gave statistical tests for (non-)exchangeability  in \cite{GeNe13} and statistical tests for radial symmetry in \cite{GeNe14}. The topic was explored also in \cite{BouChe2012,GeNeQu12}, and recently in \cite{BeaSeo2020,LyuBel2023}. Symmetries of $d$-variate  copulas for $d\ge 2$ were defined and examined by Krupskii \cite{Kru17}. The $D_4$-symmetry of bivariate copulas was studied in \cite{FuSc14} and the relationship of non-exchangeability with various concordance measures in \cite{KoKoMoOmConcordances2019}.  
Non-exchangeability and radial asymmetry of copulas are important in applications, e.g., in finance \cite{KiAcEv23, Patton}, in reliability theory \cite{Wu14}, in pricing \cite{Jas05}, and in insurance \cite{KiAs25}.

The main goal of this paper is to study the relation between non-exchangeability and  radial asymmetry. We propose the question: if a copula has a given degree of radial asymmetry, what is the range of possible values of the measure of non-exhangeability.
Our main result determines the set
\[
   \Omega_{R,T}= \{(\mu_R(C),\mu_T(C)): C \ \text{a bivariate copula}\}.
\]
It turns out that this set is a non-convex hexagon, see Figure~\ref{fig:02}.

The paper is organized as follows. Section~\ref{sec:prelim} brings preliminaries, Section~\ref{sec:main} exhibits the details of the result and some first lemmas. Section~\ref{sec:bounds} is devoted to the proof that every point of the set $\Omega_{R,T}$ belongs to the hexagon, and Section~\ref{sec:realization} to the proof that every point of the hexagon belongs to the set $\Omega_{R,T}$. We finish the paper with some concluding remarks in Section~\ref{sec:concluding}.

\section{Preliminaries} \label{sec:prelim}

Let $\II$ be the unit interval $[0,1]\subseteq\RR$ and $u_1, u_2, v_1, v_2 \in \II$ be such that $u_1 \le u_2$ and $v_1 \le v_2$. The Cartesian product of intervals $R=[u_1,u_2]\times [v_1,v_2]$ is called a {\it rectangle} in $\II^2$. Let $H:\II^2 \to \RR$ be a real function. We define the $H$-volume of $R$ as
$V_H(R)=H(u_2,v_2)-H(u_2,v_1)-H(u_1,v_2)+H(u_1,v_1)$.
A {\it bivariate copula} is a function $C:\II^2 \to \II$ with the following properties:
\begin{itemize}
    \item $C(0,v)=C(u,0)=0$ for all $u,v \in \II$ ($C$ is {\it grounded}),
    \item $C(u,1)=u$ and $C(1,v)=v$ for all $u,v \in \II$ ($C$ has {\it uniform marginals}), and
    \item $V_C(R) \geq 0$ for every rectangle $R \subseteq \II^2$ ($C$ is {\it $2-$increasing}).
\end{itemize}
Bivariate copulas are bivariate distribution functions with uniform one-dimensional mar\-gi\-nals. We can obtain any bivariate distribution function by inserting marginal distribution functions into a copula, a famous theorem by Sklar \cite{Sklar}.
The set of all bivariate copulas will be denoted by $\CC$. It is well known that this set is compact in the supremum norm. We refer to \cite{DuSe,Nels} for this and other general statements of the copula theory.

Recall that the lower and upper bound for bivariate copulas are the Fr\'{e}chet-Hoeffding bounds
$$W(u,v)=\max \{0,u+v-1\} \qquad \text{and} \qquad M(u,v)=\min \{u,v \}.$$
For any copula $C \in \CC$ the function $\widehat{C}:\II^2 \to \II$ defined by
$$\widehat{C}(u,v) = u + v - 1 +C(1-u, 1-v)$$
is again a copula, called the \emph{survival copula} of $C$. Furthermore, the function $C^T:\II^2 \to \II$ defined by
$$C^T(u,v) = C(v,u)$$
is also a copula, called the \emph{transposed  copula} of $C$.

We consider two types of asymmetry of copulas, the usual asymmetry with respect to transposition, which we call non-exchangeability to avoid confusion, and radial asymmetry.

\begin{definition}
For any copula $C$ its \emph{non-exchangeability} is defined as the supremum norm of the difference $C - C^T$, i.e.,
$$ \mu_T(C) = \max\big\{|C(u,v)-C(v,u)| \colon u,v\in \II\big\}.$$
\end{definition}

The non-exchangeability of any copula $C$ lies in the interval $[0,\frac13]$, see \cite{KleMes06,Ne07}. If $\mu_T(C) =0$, then $C=C^T$ and $C$ is called \emph{exchangeable}. 

\begin{definition}
For any copula $C$ its \emph{radial asymmetry} is defined as the supremum norm of the difference $C - \widehat{C}$, i.e.,
$$ \mu_R(C) = \max\big\{|C(u,v)-\widehat{C}(u,v)| \colon u,v\in \II\big\}.$$
\end{definition}

The radial asymmetry of any copula $C$ lies in the interval $[0,\frac13]$, see \cite{Ne07}. If $\mu_R(C) =0$, then $C=\widehat{C}$ and $C$ is called \emph{radially symmetric}. 

In  \cite{KoKoMoOm19}, the so-called maximal asymmetry function is presented; its value at a fixed point $(u, v) \in \II^2$ is given by
\begin{equation} \label{eq:dT}
    d_T(u,v) = \sup\big\{|C(u,v)-C(v,u)| \colon C\in \CC\big\}.
\end{equation}
Since $\CC$ is compact, the supremum is attained and it can be replaced by maximum. Similarly, we define
\begin{equation} \label{eq:dR}
    d_R(u,v) = \sup\big\{|C(u,v)-\widehat{C}(u,v)| \colon C\in \CC\big\}.
\end{equation}

\begin{proposition}[\cite{KleMes06,Ne07}]\label{pr:d}
The maximal asymmetry functions are
\begin{align*}
       d_T(u,v) &= \min\big\{u, v, 1-u, 1-v, |u-v|\big\}, \text{ and}\\
       d_R(u,v) &= \min\big\{u, v, 1-u, 1-v, \max\{|u-v|, |u+v-1|\}\big\}.
\end{align*}
\end{proposition}

The maximal asymmetry functions $d_T$ and $d_R$ are depicted in Figure~\ref{fig:01}.

\begin{figure} 
    \centering
    \includegraphics[width=0.35\linewidth]{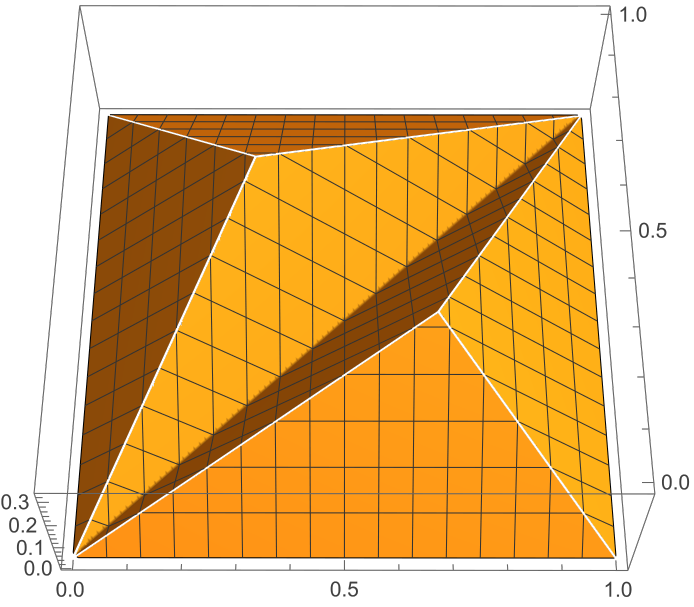} \hspace{1cm}
    \includegraphics[width=0.35\linewidth]{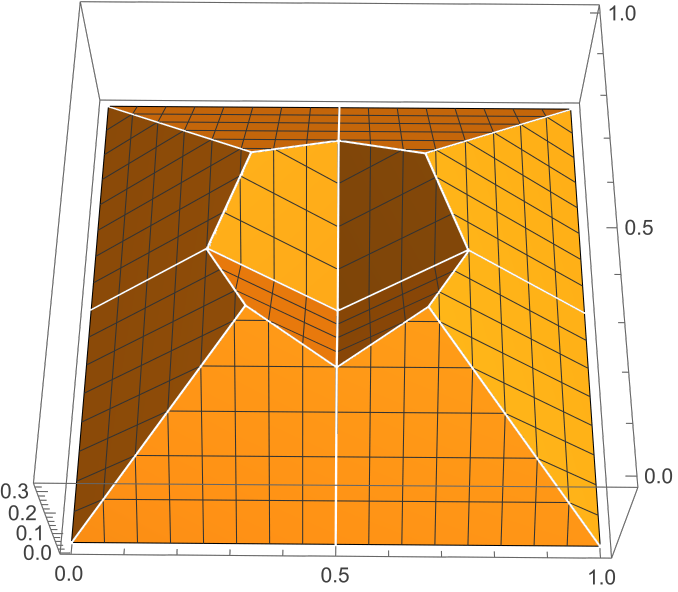}
    \caption{The maximal asymmetry functions - $d_T$ on the left and $d_R$ on the right.}
    \label{fig:01}
\end{figure}

A \emph{shuffle of $M$} is a copula
$$ C = M(n, J, \pi, \omega),$$
which is determined by a positive integer $n$, a partition
$J=\{J_1,J_2,\ldots, J_n\}$ of the interval $\II$ into $n$ pieces, where $J_i=[u_{i-1},u_i]$ and $0 =u_0 \le u_1 \le u_2 \le \ldots \le u_{n-1} \le u_n=1$, shortly written as $(n-1)$-tuple of splitting points $J=(u_1, u_2, \ldots, u_{n-1})$,
a permutation $\pi \in S_n$, written as $n$-tuple of images $\pi = (\pi(1), \pi(2), \ldots, \pi(n))$, and a mapping $\omega: \{1, 2, \ldots, n\} \to \{-1, 1\}$, written as $n$-tuple of images $\omega = (\omega(1), \omega(2), \ldots, \omega(n))$, (compare \cite[\S{3.2.3}]{Nels}). The mass of $C$ is concentrated on the main diagonals of the squares $[u_{i-1},u_i] \times  [v_{\pi(i)-1} \times v_{\pi(i)}]$ where $0= v_0 \le v_1 \le v_2\le \ldots \le v_{n-1} \le v_n = 1$ whenever $\omega(i) = 1$, and on opposite diagonals of the same squares when $\omega(i) = -1$. 

\section{The main result} \label{sec:main}

In this paper we are interested in the relation between non-exchangeability and radial asymmetry. More precisely, what can we tell about $\mu_T(C)$ for some copula $C$ if we know the value of $\mu_R(C)$, or vice versa. As we are going to see, there does not exist a copula with $\mu_R(C) = 0$ and $\mu_T(C) = \tfrac13$. On the other hand, many other combinations of values $\mu_R(C)$, $\mu_T(C)$ are possible. As our main result, we are going to determine the set of all possible pairs $(\mu_R(C),\mu_T(C)), C \in \CC$, i.e. the set
$$\Omega_{R,T} = \{(\mu_R(C),\mu_T(C)) : C \in \CC\}.$$
By observations in Section~\ref{sec:prelim} it is obvious that $\Omega_{R,T} \subseteq [0,\tfrac13]\times [0,\tfrac13]$. We will show that this set is a non-convex hexagon. 

\begin{theorem} \label{th:main}
   The set of all possible pairs $(\mu_R(C),\mu_T(C)), C \in \CC$ is a hexagon $$\Omega_{R,T} = \A\B\C\D\E\F$$ with vertices
   $$\A(0,0), \quad \B(\tfrac13,0), \quad \C(\tfrac13,\tfrac16), \quad \D(\tfrac{3}{10},\tfrac15), \quad \E(\tfrac13,\tfrac13), \quad \F(0,\tfrac14).$$
\end{theorem}

The hexagon $\A\B\C\D\E\F$ is depicted in Figure~\ref{fig:02}. 

\begin{figure}[h]
    \centering
    \includegraphics[width=0.4\linewidth]{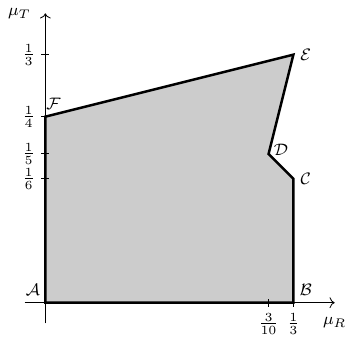}
    \caption{The set $\Omega_{R,T}$ of all possible pairs $(\mu_T(C),\mu_R(C))$.}
    \label{fig:02}
\end{figure}

In order to prove this theorem, we need to show that for any copula $C \in \CC$ the point $(\mu_R(C),\mu_T(C))$ lies in the hexagon $\A\B\C\D\E\F$, which will be done in Section~\ref{sec:bounds}, and that for any point $\PP \in \A\B\C\D\E\F$ there exists a copula $C \in \CC$ such that $(\mu_R(C),\mu_T(C)) = \PP$, which we will do in Section~\ref{sec:realization}.

Here, we start with some notation and two simple lemmas, which will be needed later. 

Let $r \in [0, \tfrac13]$. We define the following subsets of the unit square $\II^2$:
\begin{align*}
    \Delta_r^1 &= \{(u,v) \in \II^2 : u \ge r, v \ge r, u+v \le 1-r\}, \\
    \Delta_r^2 &= \{(u,v) \in \II^2 : u \ge r, v \le 1-r, v-u \ge r\}, \\
    \Delta_r^3 &= \{(u,v) \in \II^2 : u \le 1-r, v \le 1-r, u+v \ge 1+r\}, \\
    \Delta_r^4 &= \{(u,v) \in \II^2 : u \le 1-r, v \ge r, u-v \ge r\}.
\end{align*}
The sets $\Delta_r^1$, $\Delta_r^2$, $\Delta_r^3$, and $\Delta_r^4$ are triangles drawn in Figure~\ref{fig:03}. If $r > \tfrac14$, these triangles are pairwise disjoint.

\begin{figure}
    \centering
    \includegraphics[width=0.4\linewidth]{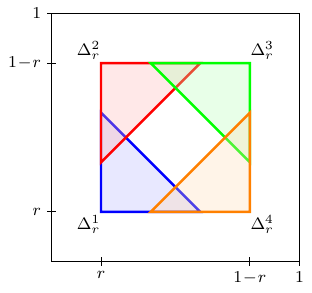}
    \hspace{1cm}
    \includegraphics[width=0.4\linewidth]{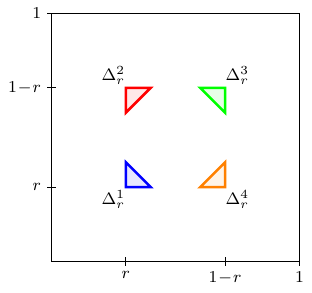}
    \caption{Triangles $\Delta_r^1$, $\Delta_r^2$, $\Delta_r^3$, and $\Delta_r^4$ for $r=\tfrac15$ (left) and $r=\tfrac{3}{10}$ (right).}
    \label{fig:03}
\end{figure}

\begin{lemma} \label{le:delta}
   \emph{(a)} Let $C \in \CC$ be a copula with $\mu_T(C) = t \in [0, \tfrac13]$. Suppose that  non-exchangeability $t$ is attained at the point $(u,v) \in \II^2$, i.e. $|C(u,v)-{C}^T(u,v)|=t$. Then $(u,v) \in \Delta_t^2 \cup \Delta_t^4$.
   
   \emph{(b)} Let $C \in \CC$ be a copula with $\mu_R(C) = r \in [0, \tfrac13]$. Suppose that radial asymmetry is attained at the point $(u,v) \in \II^2$, i.e. $|C(u,v)-\widehat{C}(u,v)|=r$. Then $(u,v) \in \Delta_r^1 \cup \Delta_r^2 \cup \Delta_r^3 \cup \Delta_r^4$.
\end{lemma}

\begin{proof}
(a) Suppose that $|C(u,v)-{C}^T(u,v)|=t\in [0,\frac13]$. Then $d_T(u,v) \ge t$, so by Proposition \ref{pr:d} we have
$$u \ge t, \quad v \ge t, \quad u \le 1-t, \quad v \le 1-t, \quad \text{and} \quad |u-v| \ge t.$$  
Now, if $u\ge v$ then $u-v\ge t$ and $v\ge t$, $u\le 1-t$. So, $(u,v)\in\Delta_t^4$. If $u\le v$ then $v-u\ge t$ and $u\ge t$, $v\le 1-t$. So, $(u,v)\in\Delta_t^2$. 

(b) Suppose that $|C(u,v)-\widehat{C}(u,v)|=r$. Then $d_R(u,v) \ge r$, so by Proposition \ref{pr:d} we have
$$u \ge r, \quad v \ge r, \quad u \le 1-r, \quad v \le 1-r, \quad \text{and} \quad \max\{|u-v|, |u+v-1|\} \ge r.$$
We consider several cases.
If $|u-v| \ge |u+v-1|$ and $u-v \ge 0$, then 
$$\max\{|u-v|, |u+v-1|\} = u-v \ge r,$$
so $(u,v) \in \Delta_r^4$. 
If $|u-v| \ge |u+v-1|$ and $u-v \le 0$, then $v-u \ge r$, so $(u,v) \in \Delta_r^2$. 
Furthermore, if $|u-v| \le |u+v-1|$ and $u+v-1 \ge 0$, then 
$$\max\{|u-v|, |u+v-1|\} = u+v-1 \ge r,$$
thus $u+v \ge 1+r$ and $(u,v) \in \Delta_r^3$. 
Finally, if $|u-v| \le |u+v-1|$ and $u+v-1 \le 0$, then $u+v \le 1-r$, so $(u,v) \in \Delta_r^1$. 
\end{proof}

\begin{lemma} \label{le:simple}
   Let $u, \alpha, \beta \in \RR$. Then $\min\{u + \alpha, \beta - u\} \le \tfrac{\alpha+\beta}{2}$.
\end{lemma}

\begin{proof}
If $u \le \tfrac{\beta-\alpha}{2}$ then $u + \alpha \le \beta - u$, so
$$\min\{u + \alpha, \beta - u\} = u + \alpha \le \tfrac{\beta-\alpha}{2} + \alpha = \tfrac{\alpha+\beta}{2}.$$
If $u \ge \tfrac{\beta-\alpha}{2}$ then $u + \alpha \ge \beta - u$, so
$$\min\{u + \alpha, \beta - u\} = \beta - u \le \beta - \tfrac{\beta-\alpha}{2} = \tfrac{\alpha+\beta}{2},$$
which finishes the proof.
\end{proof}

\section{The bounds} \label{sec:bounds}

In this section we prove that for any copula $C \in \CC$ the point $(\mu_R(C),\mu_T(C))$ lies in the hexagon $\A\B\C\D\E\F$. We first show that the point $(\mu_R(C),\mu_T(C))$ always lies on or below the edge $\E\F$, i.e. on or below the line $t = \tfrac14 r + \tfrac14$, containing points $\E(\tfrac13,\tfrac13)$ and $\F(0,\tfrac14)$.

\begin{proposition} \label{pr:zgornja1}
   For any copula $C \in \CC$ we have 
   $$\mu_T(C) \le \tfrac14 \mu_R(C) + \tfrac14.$$
\end{proposition}

\begin{proof}
    We assume that $C\in\CC$ and denote $r=\mu_R(C)\in [0,\tfrac13]$. Then $|C(u,v)-\widehat{C}(u,v)|\le r$ for any $(u,v)\in\II^2$, and hence $C(u,v)-\left(u+v-1+C(1-u,1-v)\right)\le r$ or
    \begin{equation}\label{one_star}
        C(u,v)\leq C(1-u,1-v)+r+u+v-1.
    \end{equation}

\begin{figure}[h]
    \centering
    \includegraphics[width=0.35\linewidth]{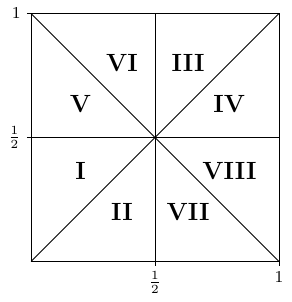}
    \caption{Regions considered in the proof of Proposition \ref{pr:zgornja1}.}
    \label{fig:04}
\end{figure}

    To prove the required inequality, we consider several cases (see Figure~\ref{fig:04}).

    \textbf{Case I.} Suppose first that $u,v\le\frac12$ and $u\leq v$. Then we have
    $$
    C(u,v)-C(v,u)\leq d_T(u,v) \leq \min\{u,v-u\}\leq\tfrac{v}{2}\leq\tfrac14\leq\tfrac{r+1}{4}.
    $$
    Here, the second inequality follows by Proposition \ref{pr:d}, the third inequality by Lemma \ref{le:simple}, the fourth inequality holds since $v\leq\frac12$, and the last one holds since $r\ge 0$.

    \textbf{Case II.} If $u,v\le\frac12$ and $u\geq v$ then, using analogous arguments as in the previous case, we have
    $$
    C(u,v)-C(v,u)\leq \min\{v,u-v\}\leq\tfrac{u}{2}\leq\tfrac14\leq\tfrac{r+1}{4}.
    $$

    \textbf{Case III.} Suppose next that $u,v\ge\frac12$ and $u\leq v$. Then we have
    $$
    C(u,v)-C(v,u)\leq d_T(u,v)\leq \min\{1-v,v-u\}\leq\tfrac{1-u}{2}\leq\tfrac14\leq\tfrac{r+1}{4},
    $$
    where the second inequality follows by Proposition \ref{pr:d}, the third inequality by Lemma \ref{le:simple}, the fourth inequality holds since $u\geq\frac12$, and the last one holds since $r\ge 0$.

    \textbf{Case IV.} If $u,v\ge\frac12$ and $u\geq v$ then, using analogous arguments as in the previous case, we have
    $$
    C(u,v)-C(v,u)\leq \min\{1-u,u-v\}\leq\tfrac{1-v}{2}\leq\tfrac14\leq\tfrac{r+1}{4}.
    $$

    \textbf{Case V.} Suppose now that $u\le\frac12$, $v\ge\frac12$ and $u+v\leq 1$. Then, using \eqref{one_star}, we have
    $$
    C(u,v)-C(v,u)\leq C(1-u,1-v)+r+u+v-1-C(v,u).
    $$
    The $1$-Lipschitz property of $C$ implies
    $$
    C(1-u,1-v)-C(v,u) \leq |(1-u)-v|+|(1-v)-u| = 2(1-u-v),
    $$
    since $u+v\leq 1$. 
    Combining both inequalities we obtain
    \begin{equation}\label{two-star}
        C(u,v)-C(v,u)\leq r+1-u-v.
    \end{equation}
    Finally, we have
    $$
    C(u,v)-C(v,u)\leq\min\{u,v-u,r+1-u-v\}\leq \min\big\{u,\tfrac{r+1-2u}{2}\big\}
    \leq\tfrac{r+1}{4},
    $$
    where the first inequality follows by Proposition \ref{pr:d} and inequality \eqref{two-star}, while the second and the third inequalities follow by Lemma \ref{le:simple}.

    \textbf{Case VI.} If $u\le\frac12$, $v\ge\frac12$ and $u+v\geq 1$ then analogous arguments as in the previous case imply that
    $$
     C(u,v)-C(v,u)\leq r+u+v-1,
    $$
    and also
    $$
    C(u,v)-C(v,u)\leq\min\{1-v,v-u,r+u+v-1\}\leq\min\big\{1-v,\tfrac{2v+r-1}{2}\big\}\leq\tfrac{r+1}{4}.
    $$

    \textbf{Case VII.} Assume that $u\geq\frac12$, $v\leq\frac12$ and $u+v\leq 1$. Then the same argument as in the Case V implies that
    $$
    C(u,v)-C(v,u)\leq r+1-u-v.
    $$
    Combing this inequality with Proposition \ref{pr:d} and Lemma \ref{le:simple} we obtain
    $$
    C(u,v)-C(v,u)\leq \min\{v,u-v,r+1-u-v\}\leq\min\big\{v,\tfrac{r+1-2v}{2}\big\}\leq\tfrac{r+1}{4}.
    $$

    \textbf{Case VIII.} In the final case $u\geq\frac12$, $v\leq\frac12$ and $u+v\geq 1$ we have 
    $$
    C(u,v)-C(v,u)\leq r+u+v-1
    $$
    using the same argument as in Case VI. Combing this inequality with Proposition \ref{pr:d} and Lemma \ref{le:simple} we have
    $$
    C(u,v)-C(v,u)\leq \min\{1-u,u-v,r+u+v-1\}\leq\min\big\{1-u,\tfrac{r-1+2u}{2}\big\}\leq\tfrac{r+1}{4}.
    $$

    We have proved that $C(u,v)-C(v,u)\leq \tfrac{r+1}{4}$ holds for any $(u,v)\in\II^2$. Now,
    \begin{align*}
    \mu_T(C) &= \max\big\{|C(u,v)-C(v,u)| \colon u,v\in \II\big\} \\
    &= \max\big\{C(u,v)-C(v,u) \colon u,v\in \II\big\} \\
    &\leq \tfrac{r+1}{4} = \tfrac14 \mu_R(C) +  \tfrac14,      
    \end{align*}
    which finishes the proof.
\end{proof}

Next we show that the points in the interior of the triangle $\C\D\E$ are not attained by some copula. In order to do this we assume that $\mu_R(C) \ge \tfrac14$ and consider two cases, where radial asymmetry $\mu_R$ is attained.

\begin{proposition} \label{pr:spodnja}
   Suppose that for some copula $C \in \CC$ it holds that $\mu_R(C) = r \in [\tfrac14, \tfrac13]$ and that radial asymmetry $\mu_R$ is attained at a point $(u,v) \in \Delta_r^2$. Then $$\mu_T(C) \ge 4 \mu_R(C) - 1.$$
\end{proposition}

\begin{proof}
   By assumptions of the proposition we have $|C(u,v)-\widehat{C}(u,v)| = r$ for some point $(u,v) \in \Delta_r^2$. Assume first that $C(u,v)-\widehat{C}(u,v) \ge 0$. Then 
   $$ C(u,v) = \widehat{C}(u,v) + r = C(1-u,1-v) + u + v - 1 + r.$$
   
   We consider two cases. Suppose that $u+v \le 1$. Then 
   $$ C(u,v) \le C(1-v,v) \le C(1-v, 1-u),$$
   since vertical and horizontal sections of a copula are non-decreasing. Now,
   \begin{align*}
      \mu_T(C) &\ge C(1-v, 1-u) - C(1-u,1-v) \\
      &\ge C(u,v) - C(1-u,1-v) \\
      &= C(1-u,1-v) + u + v - 1 + r - C(1-u,1-v) \\
      &= u + v - 1 + r \\
      &\ge 2u + 2r - 1  \\
      &\ge 4r - 1, 
   \end{align*}
   since $(u,v) \in \Delta_r^2$, so that $v \ge u+r$ and $u \ge r$. 

   Suppose now that $u+v \ge 1$. Then 
   $$ \widehat{C}(1-u,1-v) \le \widehat{C}(v,1-v) \le \widehat{C}(v,u)$$
   and 
   \begin{align*}
      \mu_T(C) &\ge C(1-v, 1-u) - C(1-u,1-v) \\ 
      &= \widehat{C}(v,u) + 1 - v - u - C(1-u,1-v) \\
      &\ge \widehat{C}(1-u,1-v) + 1 - v - u - C(1-u,1-v) \\
      &= C(u,v) + 2 - 2u - 2v - C(1-u,1-v) \\
      &= C(1-u,1-v) + u + v - 1 + r - C(1-u,1-v) \\
      &= r + 1 - u - v  \\
      &\ge 4r - 1, 
   \end{align*}
    since $(u,v) \in \Delta_r^2$, so that $-u \ge r - v$ and $-v \ge r - 1$.  
    In both cases we have proved that  $\mu_T(C) \ge 4r-1 = 4 \mu_R(C) - 1.$

   In the case that $C(u,v)-\widehat{C}(u,v) \le 0$ we have $\widehat{C}(u,v) - C(u,v) \ge 0$. We interchange the roles of $C$ and $\widehat{C}$ to deduce 
   $\mu_T(\widehat{C}) \ge 4r-1$ in the same way. Now, 
   \begin{align*}
      \mu_T(\widehat{C}) &= \max\big\{|\widehat{C}(u,v)-\widehat{C}(v,u)| \colon u,v\in \II\big\} \\
      &= \max\big\{|u+v-1+C(1-u,1-v)-(v+u-1+C(1-v,1-u))| \colon u,v\in \II\big\} \\
      &= \max\big\{|C(1-u,1-v)-C(1-v,1-u))| \colon u,v\in \II\big\} \\
      &= \max\big\{|C(u,v)-C(v,u))| \colon u,v\in \II\big\} = \mu_T(C),
   \end{align*}
   so again $\mu_T(C) \ge 4r-1 = 4 \mu_R(C) - 1$, which finishes the proof.
\end{proof}

\begin{proposition} \label{pr:zgornja2}
   Suppose that for some copula $C \in \CC$ it holds that $\mu_R(C) = r \in [\tfrac14, \tfrac13]$ and that radial asymmetry $\mu_R$ is attained at a point in $\Delta_r^1$. Then $$\mu_T(C) \le \tfrac12 - \mu_R(C).$$
\end{proposition}

\begin{proof}
    Suppose that $\mu_R(C)=r$ is attained at a point $(u_0,v_0)\in\Delta_r^1$. Hence, we have that $r=C(u_0,v_0)-\widehat{C}(u_0,v_0)$. We write $s=C(u_0,v_0)-r$.  So,
    $$
    s=\widehat{C}(u_0,v_0)\geq 0.
    $$
    
    We would like to show that $C(u,v)-C(v,u)\leq \frac12-r$ for all $(u,v)\in\II^2$. To do so, we consider several cases (see Figure~\ref{fig:05}).

    \textbf{Case I.} Assume first that $u\leq v_0$ and $v\leq u_0$. Then the $1$-Lipshitz property implies that 
    $$
    C(u_0,v_0)-C(v,u)\leq u_0-v+v_0-u,
    $$
    and hence
    \begin{equation*}\label{3-star}
        C(v,u)\geq C(u_0,v_0)+u+v-u_0-v_0=r+s+u+v-u_0-v_0.
    \end{equation*}
    Since $C(u,v)\leq u$, $s\geq 0$ and $u_0+v_0\leq 1-r$, this inequality implies that
    $$
    C(u,v)-C(v,u)\leq u_0+v_0-r-s-v\leq u_0+v_0-r-v\leq 1-2r-v.
    $$
    Combing the latter inequality with Proposition \ref{pr:d} and Lemma \ref{le:simple}, respectively, we obtain
    $$
    C(u,v)-C(v,u)\leq \min\{v,1-2r-v\}\leq \tfrac12-r.
    $$

\begin{figure}[h]
    \centering
    \includegraphics[width=0.4\linewidth]{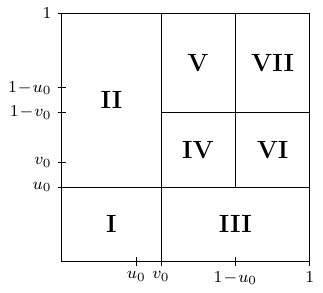}
    \caption{Regions considered in the proof of Proposition \ref{pr:zgornja2} for $u_0=\tfrac{3}{10}$, $v_0=\tfrac25$.}
    \label{fig:05}
\end{figure}

    \textbf{Case II.} Suppose that $u\leq v_0$ and $v\geq u_0$. From these assumptions and the $1$-Lipschitz property it follows that 
    $$
    C(v,u)\geq C(v,v_0)-(v_0-u)\geq C(u_0,v_0)+u-v_0=r+s+u-v_0.
    $$
    Since $C(u,v)\leq u$, $s\geq 0$ and $v_0\leq1-r-u_0\leq1-2r\leq\frac{1}{2}$, the above inequality implies that
    $$
    C(u,v)-C(v,u)\leq v_0-r-s\leq v_0-r\leq \tfrac12-r.
    $$

    \textbf{Case III.} Next, we consider the case $u\geq v_0$ and $v\leq u_0$. The $1$-Lipschitz property implies that $C(u_0,u)-C(v,u)\leq u_0-v$. Hence, it follows that 
    $$
    C(v,u)\geq C(u_0,u)-u_0+v\geq C(u_0,v_0)-u_0+v = r+s-u_0+v.
    $$
    Since $C(u,v)\leq v$, $s\geq 0$ and $u_0\leq\frac{1}{2}$ (as above) we obtain
    $$
    C(u,v)-C(v,u)\leq u_0-r-s\leq u_0-r\leq \tfrac12-r.
    $$

    \textbf{Case IV.} Now, we choose $(u,v)\in [v_0,1-u_0]\times [u_0,1-v_0]$. Observe that then $C(v,u)\geq C(u_0,v_0)=r+s$. 

    Assume first that $u\leq v$. Then $C(u,v)\leq u$ and Proposition \ref{pr:d} imply that
    $$
    C(u,v)-C(v,u)\leq u-r-s
    $$
    and
    $$
    C(u,v)-C(v,u)\leq v-u\leq 1-v_0-u,
    $$
    respectively. Combing these two inequalities with Lemma \ref{le:simple} and using that $s\ge 0$, $v_0\ge r$, we obtain
    $$
    C(u,v)-C(v,u)\leq \min\{u-r-s,1-v_0-u\}\leq \tfrac{1-v_0-r-s}{2}\leq \tfrac{1-v_0-r}{2}\leq \tfrac12-r.
    $$

    If $u\geq v$ then using analogous arguments as above we get first that
    $$
    C(u,v)-C(v,u)\leq v-r-s
    $$
    and
    $$
    C(u,v)-C(v,u)\leq u-v \leq 1-u_0-v.
    $$
    Then, combing these two inequalities with Lemma \ref{le:simple} and using that $s\ge 0$, $u_0\ge r$, we obtain
    $$
    C(u,v)-C(v,u)\leq \min\{v-r-s,1-u_0-v\}\leq \tfrac{1-u_0-r-s}{2}\leq \tfrac{1-u_0-r}{2}\leq \tfrac12-r.
    $$
    
    \textbf{Case V.} Assume that $(u,v)\in [v_0,1-u_0]\times [1-v_0,1]$. Then our assumptions and the $1$-Lipschitz property imply that
    $$
    C(u,v)\leq C(1-u_0,v)\leq C(1-u_0,1-v_0) +v-(1-v_0)=s+v-u_0,
    $$
    since $C(1-u_0,1-v_0) = s + 1 -u_0 -v_0$, and
    $$
    C(v,u)\geq C(u_0,u)\geq C(u_0,v_0)=r+s.
    $$
    It follows that 
    $$
    C(u,v) - C(v,u) \leq v-u_0 - r.
    $$
    Furthermore, since $C(u,v)\leq u$ and $C(v,u)\geq W(v,u)\geq u+v-1$, we have
    $$
    C(u,v) - C(v,u) \leq 1-v.
    $$
    Using Lemma \ref{le:simple} this implies 
    $$
    C(u,v) - C(v,u) \leq \min\{v-u_0-r,1-v\} \leq \tfrac{1-r-u_0}{2}\leq \tfrac12-r.
    $$
    
    \textbf{Case VI.} Now we take $(u,v)\in [1-u_0,1]\times [u_0,1-v_0]$. Then $C(u,v)\le v$ and $C(v,u)\ge u+v-1$ imply $C(u,v) - C(v,u) \leq 1-u$. Furthermore,
    $$
    C(u,v)\leq C(u,1-v_0)\leq C(1-u_0,1-v_0)+u-1+u_0=s+u-v_0
    $$
    and
    $$
    C(v,u)\ge C(u_0,v_0)=r+s,
    $$
    so $C(u,v) - C(v,u) \leq u-v_0-r$. Using Lemma \ref{le:simple} we obtain
    $$
    C(u,v)-C(v,u) \leq \min\{u-v_0-r,1-u\} \leq \tfrac{1-r-v_0}{2}\leq \tfrac12-r.
    $$

    \textbf{Case VII.} It remains to consider the case $(u,v)\in [1-u_0,1]\times [1-v_0,1]$.  

    Assume first that $u\leq v$. Then, $C(u,v)\leq u$, $C(v,u)\geq u+v-1$ and therefore $C(u,v)-C(v,u)\leq 1-v$. Since $(u_0,v_0)\in\Delta^1_r$ we have $u_0\leq1-r-v_0\leq1-2r$. Then Proposition \ref{pr:d} and Lemma \ref{le:simple} imply that
    $$
    C(u,v)-C(v,u)\leq \min\{v-u,1-v\}\leq \tfrac{1-u}{2}\leq \tfrac{u_0}{2}\leq \tfrac12-r.
    $$

    If $u\geq v$ then using analogous arguments as above we get that $C(u,v)-C(v,u)\leq1-u$, and
    $v_0\leq 1-2r$.
    Then, using Proposition \ref{pr:d} and Lemma \ref{le:simple} we obtain
    $$
    C(u,v)-C(v,u)\leq \min\{u-v,1-u\}\leq  \tfrac{1-v}{2} \leq \tfrac{v_0}{2}\leq \tfrac12-r.
    $$

    We have proved that $C(u,v)-C(v,u)\leq \tfrac{1}{2}-r$ holds for any $(u,v)\in\II^2$. Now,
    \begin{align*}
    \mu_T(C) &= \max\big\{|C(u,v)-C(v,u)| \colon u,v\in \II\big\} \\
    &= \max\big\{C(u,v)-C(v,u) \colon u,v\in \II\big\} \\
    &\leq \tfrac{1}{2}-r = \tfrac12- \mu_R(C),      
    \end{align*}
    which finishes the proof.
\end{proof}

We collect our findings in the main theorem of this section. 

\begin{theorem} \label{th:bounds}
   For any copula $C \in \CC$ the point $\PP = (\mu_R(C),\mu_T(C))$ lies in the hexagon $\A\B\C\D\E\F$.
\end{theorem}

\begin{proof}
   Since  $\mu_R(C) \in [0, \tfrac13]$ and $\mu_T(C) \in [0, \tfrac13]$, we have $\PP \in [0, \tfrac13] \times [0, \tfrac13]$. Proposition \ref{pr:zgornja1} shows that point $\PP$ lies in the quadrilateral $\A\B\E\F$. Note that the lines $t = 4r - 1$ and $t = \tfrac12 - r$ intersect at the point $\D(\tfrac{3}{10}, \tfrac15)$, while these lines intersect the vertical line $r= \tfrac13$ at the points $\C(\tfrac13, \tfrac16)$ and $\E(\tfrac13, \tfrac13)$, respectively. 
   
   Now suppose that $\mu_R(C) = r > \tfrac{3}{10} > \tfrac14$ and the radial asymmetry of copula $C$ is attained at the point $(u,v) \in \II^2$. By Lemma \ref{le:delta} we have $(u,v) \in \Delta_r^1 \cup \Delta_r^2 \cup \Delta_r^3 \cup \Delta_r^4$. If $(u,v) \in \Delta_r^1$ then $\mu_T(C) \le \tfrac12 - \mu_R(C)$ by Proposition \ref{pr:zgornja2}, so the point $\PP$ lies on or below the line $t = \tfrac12 - r$. If $(u,v) \in \Delta_r^2$ then $\mu_T(C) \ge 4 \mu_R(C) - 1$ by Proposition \ref{pr:spodnja}, so the point $\PP$ lies on or above the line $t = 4r-1$. 
   
   If $(u,v) \in \Delta_r^3$ then  
   \begin{align*}
       |C(1-u,1-v)-\widehat{C}(1-u,1-v)| &= |1-u-v+\widehat{C}(u,v)-(1-u-v+C(u,v))| \\
       &= |\widehat{C}(u,v)-C(u,v)| \\
       &= |C(u,v)-\widehat{C}(u,v)| = \mu_R(C),
   \end{align*}
   so the radial asymmetry of copula $C$ is also attained at the point $(1-u, 1-v) \in \Delta_r^1$. Thus again the point $\PP$ lies on or below the line $t = \tfrac12 - r$ by Proposition \ref{pr:zgornja2}. Finally, if $(u,v) \in \Delta_r^4$ then  
   the radial asymmetry of copula $C$ is also attained at the point $(1-u, 1-v) \in \Delta_r^2$. So the point $\PP$ lies on or above the line $t = 4r-1$ by Proposition \ref{pr:spodnja}. 
   
   It follows that the point $\PP$ does not lie in the interior of the triangle $\C\D\E$ or in the interior of the edge $\C\E$, since it lies either on or below the edge $\C\D$ or on or above the edge $\D\E$. This implies that $\PP \in \A\B\C\D\E\F$, which finishes the proof.
\end{proof}

\section{Realizations} \label{sec:realization}

In this section we show that any point $\PP$ in hexagon $\A\B\C\D\E\F$ is attained by some copula $C \in \CC$. We first show that this is true for the edges $\B\C$, $\C\D$, $\D\E$, and $\E\F$.

\begin{proposition} \label{pr:BC}
    Let $\PP$ be a point on line segment $\B\C$, where $\B(\tfrac13,0), \C(\tfrac13,\tfrac16)$. Then there exists a copula $C \in \CC$, such that $(\mu_R(C),\mu_T(C)) = \PP$.
\end{proposition}

\begin{figure}[h]
    \centering
    \includegraphics[width=0.6\linewidth]{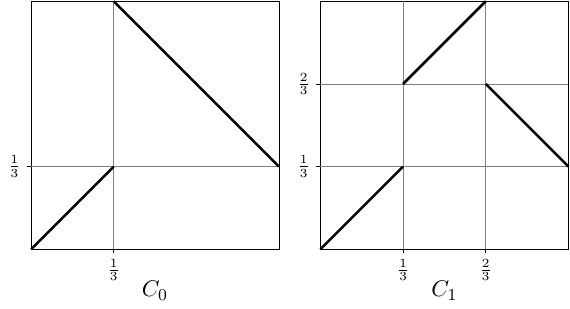} \\[4mm]
    \includegraphics[width=0.6\linewidth]{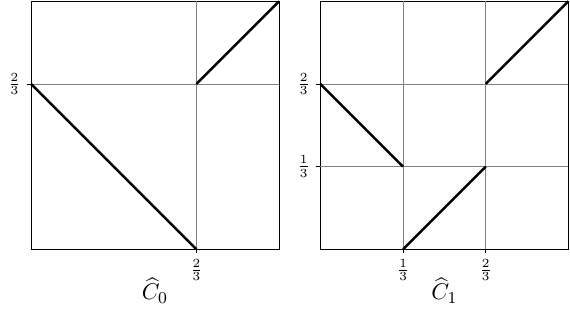}
    \caption{Mass distribution of copulas $C_0$, $\widehat{C}_0$, $C_1$, and $\widehat{C}_1$.}
    \label{fig:06}
\end{figure}

\begin{proof}
Let $C_0$ and $C_1$ be shuffles of $M$
$$C_0 = M(2, (\tfrac13), (1, 2), (1, -1)), \qquad C_1 = M(3, (\tfrac13, \tfrac23), (1, 3, 2), (1, 1, -1)).$$
Then 
$$C_0^T = M(2, (\tfrac13), (1, 2), (1, -1)) = C_0, \qquad C_1^T = M(3, (\tfrac13, \tfrac23), (1, 3, 2), (1, -1, 1)),$$
and
$$\widehat{C}_0 = M(2, (\tfrac23), (1, 2), (-1, 1)), \qquad \widehat{C}_1 = M(3, (\tfrac13, \tfrac23), (2, 1, 3), (-1, 1, 1)).$$

The mass distributions of copulas $C_0$, $\widehat{C}_0$, $C_1$, and $\widehat{C}_1$ are depicted in Figure~\ref{fig:06}. Note that $C_0$ and $C_1$ differ only in the square $[\tfrac13, \tfrac23] \times [\tfrac23, 1]$. Furthermore, for any $\lambda \in \II$ let 
$$ C_\lambda = (1-\lambda)C_0 + \lambda C_1. $$
We have
$$ C_\lambda(\tfrac13, \tfrac13) - \widehat{C}_\lambda(\tfrac13, \tfrac13) = (1-\lambda)\cdot\tfrac13 + \lambda \cdot\tfrac13  - \big((1-\lambda)\cdot 0 + \lambda \cdot 0\big) = \tfrac13,$$
so $\mu_R(C_\lambda) \ge \frac13$. Since $\mu_R(C) \le \frac13$ for any copula $C$, this implies $\mu_R(C_\lambda) = \frac13$ for any $\lambda \in \II$. Next, note that $C_\lambda$ and $C_\lambda^T$ differ only in the squares $[\tfrac13, \tfrac23] \times [\tfrac23, 1]$ and $[\tfrac23, 1] \times [\tfrac13, \tfrac23]$, and that the non-exchangeability of $C_\lambda$ is attained in the middle of these squares, i.e. at the point $(\tfrac12, \tfrac56)$, where
$$ C_\lambda(\tfrac12, \tfrac56) - C_\lambda(\tfrac56, \tfrac12) = (1-\lambda)\cdot\tfrac13 + \lambda \cdot\tfrac12  - \big((1-\lambda)\cdot \tfrac13 + \lambda \cdot  \tfrac13\big)  = \lambda \cdot \tfrac16,$$
so $\mu_T(C_\lambda) = \tfrac{\lambda}{6}$. Now, point $\PP$ on line segment $\B\C$ has coordinates $\PP(\tfrac13, p)$, where $p \in [0, \tfrac16]$, 
so $\PP = (\mu_R(C_{6p}),\mu_T(C_{6p}))$.
\end{proof}

\begin{proposition} \label{pr:CD}
    Let $\PP$ be a point on line segment $\C\D$, where $\C(\tfrac13,\tfrac16), \D(\tfrac{3}{10},\tfrac15)$. Then there exists a copula $C \in \CC$, such that $(\mu_R(C),\mu_T(C)) = \PP$.
\end{proposition}

\begin{proof}
Let $r \in [\tfrac{3}{10}, \tfrac13]$ and let $D_r$ be shuffle of $M$
$$D_r = M(3, (r, 1-r), (1, 3, 2), (1, 1, -1)).$$
Then 
$$D_r^T = M(3, (r, 2r), (1, 3, 2), (1, -1, 1)) \quad \text{and} \quad \widehat{D}_r = M(3, (r, 1-r), (2, 1, 3), (-1, 1, 1)).$$

\begin{figure}[h]
    \centering
    \includegraphics[width=0.6\linewidth]{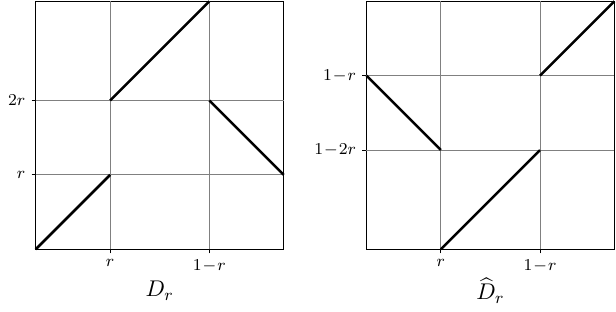}
    \caption{Mass distribution of copulas $D_r$ and $\widehat{D}_r$ for $r=\tfrac{3}{10}$.}
    \label{fig:07}
\end{figure}

The mass distributions of copulas $D_r$ and $\widehat{D}_r$ are depicted in Figure~\ref{fig:07}. For any $r \in [\tfrac{3}{10}, \tfrac13]$ we have
$$ D_r(r,r) - \widehat{D}_r(r,r) = r,$$
so, $\mu_R(D_r) \ge r$. In order to prove that $\mu_R(D_r) \le r$ we need to show that $|D_r(u,v) - \widehat{D}_r(u,v)| \le r$ for any $(u,v) \in \Delta_r^1 \cup \Delta_r^2 \cup \Delta_r^3 \cup \Delta_r^4$ by Lemma \ref{le:delta} (b).

The triangle $\Delta_r^1$ has vertices $(r,r)$, $(r,1-2r)$, and $(1-2r,r)$. Since $1-2r < 2r \le 1-r$, we have $\Delta_r^1 \subset [r, 1-r] \times [r, 2r]$, so $D_r(u,v) = r$ and
$$ 0\le \widehat{D}_r(u,v) = \widehat{D}_r(u,v) - \widehat{D}_r(r,r) \le u-r + v - r \le 1-3r < r$$
by $1$-Lipschitz condition and since $u+v \le 1-r$. It follows that
$$ 0 < D_r(u,v) - \widehat{D}_r(u,v) \le r$$
for any $(u,v) \in \Delta_r^1$.  The triangle $\Delta_r^2$ has vertices $(r,2r)$, $(r,1-r)$, and $(1-2r,1-r)$.  If $(u,v) \in \Delta_r^2$ then $D_r(u,v) = u$ and $D_r(1-u,1-v) = r$, so
$$ D_r(u,v) - \widehat{D}_r(u,v) = u - (u+v-1+ D_r(1-u,1-v)) = 1-v-r \le 1-u-2r \le 1-3r \le r,$$
since $v\ge u+r$, $u\ge r$, and $r\ge \tfrac{3}{10}$. Furthermore, $1-v-r\ge 0$, so
$|D_r(u,v) - \widehat{D}_r(u,v)| \le r$. 
If $(u,v) \in \Delta_r^3 \cup \Delta_r^4$ then $(1-u,1-v) \in \Delta_r^1 \cup \Delta_r^2$, so 
$$|D_r(u,v) - \widehat{D}_r(u,v)| = |D_r(1-u,1-v) - \widehat{D}_r(1-u,1-v)| \le r$$
by above. 

We have thus proved that $\mu_R(D_r) = r$. Since radial asymmetry is attained at a point in $\Delta_r^1$, we have 
$$\mu_T(D_r) \le \tfrac12 - \mu_R(D_r) = \tfrac12-r$$ 
by Proposition \ref{pr:zgornja2}. Furthermore, since the point $(\tfrac12, \tfrac12+r)$ lies on the main diagonal of the square $[r,1-r]\times[2r,1]$ and the point $(\tfrac12+r, \tfrac12)$ lies on the opposite diagonal of the square $[1-r,1]\times[r,2r]$, we have
$$ D_r(\tfrac12, \tfrac12+r) - D_r(\tfrac12+r, \tfrac12)= \tfrac12 -r,$$
so $\mu_T(D_r) = \tfrac12 -r$.
Now, point $\PP$ on line segment $\C\D$ has coordinates $\PP(r, \tfrac12 -r)$, where $r \in [\tfrac{3}{10}, \tfrac13]$, 
so $\PP = (\mu_R(D_r),\mu_T(D_r))$.
\end{proof}

\begin{proposition} \label{pr:DE}
    Let $\PP$ be a point on line segment $\D\E$, where $\D(\tfrac{3}{10},\tfrac15), \E(\tfrac13,\tfrac13)$. Then there exists a copula $C \in \CC$, such that $(\mu_R(C),\mu_T(C)) = \PP$.
\end{proposition}

\begin{proof}
    We choose $r\in [\tfrac{3}{10},\tfrac13]$ and denote $t=4r-1$. Observe that then 
    $$
    0\leq 1-3r<r<2r\le 1-r<1.
    $$ 
    Let us consider the shuffle of $M$ 
    $$
    E_r=M(5,(1-3r,r,2r,1-r),(1,3,5,4,2),(1,1,1,1,1)).
    $$ 
    Then
    $$
    E_r^T=M(5,(1-3r,1-2r,2r,1-r),(1,5,2,4,3),(1,1,1,1,1))
    $$
    and 
    $$
    \widehat{E}_r=M(5,(r,1-2r,1-r,3r),(4,2,1,3,5),(1,1,1,1,1)).
    $$
    The mass distributions of copulas $E_r$ and $\widehat{E}_r$ are depicted in Figure~\ref{fig:08}.

\begin{figure}[h]
    \centering
    \includegraphics[width=0.6\linewidth]{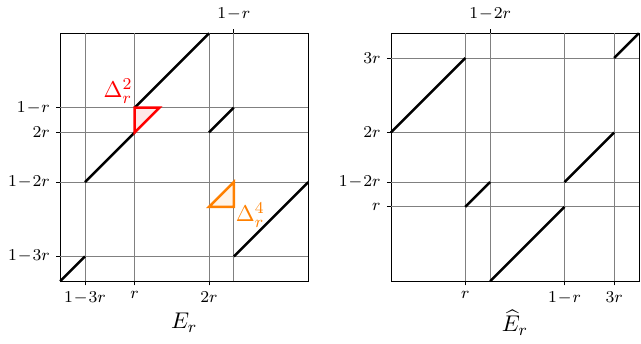}
    \caption{Mass distribution of copulas $E_r$ and $\widehat{E}_r$ for $r=\tfrac{3}{10}$.}
    \label{fig:08}
\end{figure}

    We claim that $\mu_R(E_r)=r$ and $\mu_T(E_r)=t$.

    Note first that $E_r(r,2r)=r$ and $\widehat{E}_r(r,2r)=0$. It follows that $\mu_R(E_r)\ge r$. 
    To show that $\mu_R(E_r)\leq r$ we have to show that 
    $$
    |E_r(u,v)-\widehat{E}_r(u,v)|\leq r
    $$
    for all $(u,v)\in\Delta^1_r\cup\Delta^2_r$. This suffices by Lemma \ref{le:delta} (b) and symmetric position of the $\Delta_r^i$, $i=1,2,3,4$, sets.

    Assume first that $(u,v)\in\Delta_r^1$. Note that then $u\ge r$, $v\ge r$ and $u+v\le 1-r$. Since $u\le r$ and $v\leq 1-r-u\leq 1-2r$ we have 
    \begin{equation}\label{E_r=1-3r}
    E_r(u,v)=1-3r.
     \end{equation}
    Since both $u$ and $v$ lie in the interval $[r, 1-2r]$ it follows that both $1-u$ and $1-v$ belong to the interval $[2r,1-r]$. Hence, we have 
    $$
    r\leq E_r(1-u,1-v)\leq 1-2r.
    $$
    Then 
    $$
    0\leq \widehat{E}_r(u,v) = u + v - 1 + E_r(1-u,1-v) \leq  1-r - 1 + 1-2r = 1-3r,
    $$ 
    since $\widehat{E}_r$ is a copula and $u+v \le 1-r$. Using these inequalities and \eqref{E_r=1-3r} we obtain
    $$
    0\leq E_r(u,v)-\widehat{E}_r(u,v)\leq 1-3r < r.
    $$

    Assume now that $(u,v)\in\Delta_r^2$, so that $u\ge r$, $v\le 1-r$ and $v\ge u+r$. Since $r\ge\frac{3}{10}$ these relations imply that $r\le u\le 2r$ and $2r\le v\le 1-r$. Then it is easy to observe that
    \begin{equation}\label{E_r=r}
    E_r(u,v)=r.
     \end{equation}
    Also, since $1-3r<1-2r\leq 1-u\leq 1-r$ and $1-3r<r\leq 1-v\leq 1-2r$ it follows that  
    $$
     E_r(1-u,1-v)=1-3r.
    $$
    Then 
    $$
    \widehat{E}_r(u,v)=u+v-3r.
    $$ 
    Using this equality and \eqref{E_r=r} we obtain
    $$
    E_r(u,v)-\widehat{E}_r(u,v)\leq E_r(u,v)= r
    $$
    and
    $$
    E_r(u,v)-\widehat{E}_r(u,v)=4r-u-v\geq 3r-1 > -r.
    $$
    This completes the proof of the claim that $\mu_R(E_r)\leq r$, and thus $\mu_R(E_r)=r$.

    Next, we consider $\mu_T(E_r)$. Note that $E_r(r,2r)=r$ and $E_r(2r,r)=1-3r$, since $1-3r < r \le 1-2r$. Then $E_r(r,2r)-E_r(2r,r)=4r-1=t$ and thus $\mu_T(E_r)\ge t$. 
    
    To show that $\mu_T(E_r)\leq t$ it suffices to show that 
    \begin{equation}\label{mu_T-bound}
    |E_r(u,v)-{E_r^T}(u,v)|\leq t
    \end{equation}
    for all $(u,v)\in\Delta^2_t$. Namely, function $d_T(u,v)$ exceeds $t$ only on $\Delta_t^2\cup\Delta_t^4$ by Lemma \ref{le:delta}~(a).

    To prove \eqref{mu_T-bound}, we consider several cases depending on the position of $(u,v)$ in $\Delta_t^2$ (see Figure~\ref{fig:09}). Note that we have $u\ge t$, $v\le 1-t$ and $v\ge u+t$. 

    \begin{figure}
    \centering
    \includegraphics[width=0.6\linewidth]{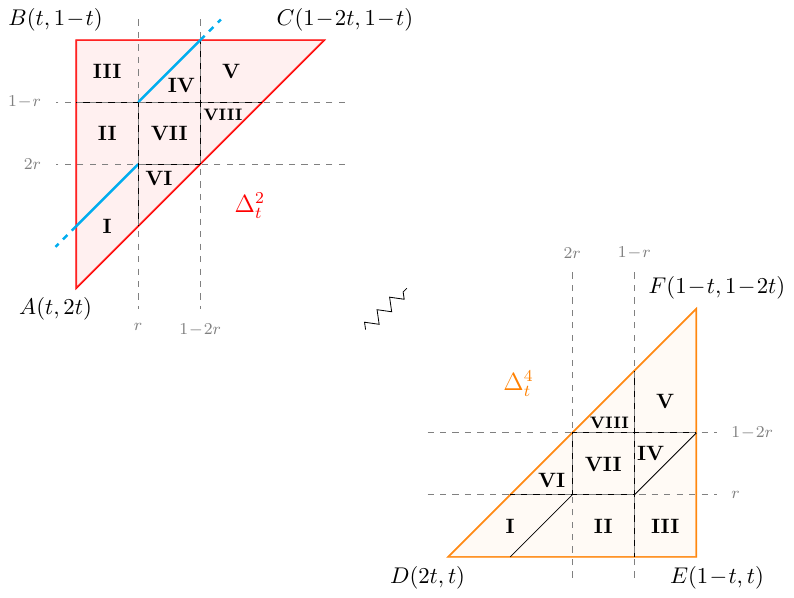}
    \caption{Regions considered in the proof of Proposition \ref{pr:DE} for $r=\tfrac{3}{10}$ and $t=\tfrac15$. The cyan lines indicate the mass of the copula $E_r$ inside triangle $ABC$. There is no mass of $E_r$  inside triangle $DEF$.}
    \label{fig:09}
\end{figure}

    The vertices of the triangle $\Delta_t^2$ are $A(t,2t)$, $B(t,1-t)$ and $C(1-2t,1-t)$. Since $r\in [\frac{3}{10},\frac13]$ it follows that $1-3r<t\le r$, $1-2r\le 2t \le 2r$ and so the vertex $(t,2t)$ lies in the square $R_2=[1-3r,r]\times [1-2r,2r]$. Since $2t-t=t\le r$ it lies below the main diagonal of the square $R_2$. The same relations imply also that the vertex $(1-2t,1-t)$ lies below the main diagonal in the square $R_3=[r,2r]\times [1-r,1]$. Namely, we have $1-2t\le 1-2r$ and $1-t\ge 1-r$.

    The vertices of the triangle $\Delta_t^4$ are $D(2t,t)$, $E(1-t,t)$ and $F(1-t,1-2t)$. Since $1-r\le 1-t$, $t \le 1-2r$ and $t-(1-t)=2t-1\ge -2r$, it follows that the vertex $(1-t,t)$ lies above the main diagonal in the square $R_5=[1-r,1]\times [1-3r,1-2r]$.

    \textbf{Case I.} Assume first that $t\le u\le r$ and $u+t\le v\le u+r$. Then $E_r(u,v)=1-3r+v-(1-2r)=v-r$ and $E_r(v,u)=1-3r$. So, $E_r(u,v)-E_r(v,u)=v+2r-1.$ It follows that
    \begin{align*}
    0&\le 10r-3=2t+2r-1\leq u+t+2r-1\\
    &\leq v+2r-1=E_r(u,v)-E_r(v,u)\\ 
    &\leq u+3r-1\leq 4r-1=t.
    \end{align*}

    \textbf{Case II.} Take now $t\le u\le r$ and $u+r\le v\le 1-r$. We find that $E_r(u,v)=u$ and $E_r(v,u)=1-3r$. So, $E_r(u,v)-E_r(v,u)=u+3r-1.$ Then it follows that
    \begin{align*}
    0&< 7r-2=t+3r-1\\
    &\leq u+3r-1=E_r(u,v)-E_r(v,u)\\ 
    &\leq 4r-1=t.
    \end{align*}
    
    \textbf{Case III.} 
    Take now $1-r\le v\le 1-t$ and $t\le u\le v+2r-1$. We find that $E_r(u,v)=u$ and $E_r(v,u)=1-3r+v-(1-r)=v-2r$. So, $E_r(u,v)-E_r(v,u)=u-v+2r.$ Then it follows that
    \begin{align*}
    0&\leq 10r-3= t+t-1+2r\\
    &\leq u-v+2r=E_r(u,v)-E_r(v,u)\\ 
    &\leq 4r-1=t.
    \end{align*}
    
    \textbf{Case IV.} Next, we assume that $1-r\le v\le 1-t$ and $v-(1-2r)\le u\le 1-2r$. In this case, we find that $E_r(u,v)=r+v-(1-r)=v+2r-1$ and $E_r(v,u)=v-2r$. So, $$E_r(u,v)-E_r(v,u)=4r-1=t.$$ 
        
    \textbf{Case V.} Assume now that $1-r\le v\le 1-t$ and $1-2r\le u\le v-t$. Then $E_r(u,v)=v+2r-1$ and $E_r(v,u)=1-3r+v-(1-r)+u-(1-2r)=u+v-1$. So, $E_r(u,v)-E_r(v,u)=2r-u$. Then, it follows that
    \begin{align*}
    0&\leq 10r-3=2r-((1-t)-t)\\
    &\leq 2r-(v-t)\\
    &\leq 2r-u=E_r(u,v)-E_r(v,u)\\ 
    &\leq 4r-1=t.
    \end{align*}
    
    \textbf{Case VI.} Now, we take $r\le u\le 1-2r$ and $u+t\le v\le 2r$. Then $E_r(u,v)=v-r$ and $E_r(v,u)=1-3r$. So, $E_r(u,v)-E_r(v,u)=v+2r-1$. Then, it follows that
    \begin{align*}
    0&< 7r-2\leq u+t+2r-1\\
    &\leq v+2r-1=E_r(u,v)-E_r(v,u)\\ 
    &\leq 4r-1=t.
    \end{align*}
    
    \textbf{Case VII.} Next, we suppose that $r\le u\le 1-2r$ and $2r\le v\le 1-r$. Then $E_r(u,v)=r$ and $E_r(v,u)=1-3r$. So, 
    $$E_r(u,v)-E_r(v,u)=4r-1=t.$$ 
    
    \textbf{Case VIII.} Finally, assume that $2r\le v\le 1-r$ and $1-2r\le u\le v-t$. Then $E_r(u,v)=r$ and $E_r(v,u)=1-3r+u-(1-2r)=u-r$. So, $E_r(u,v)-E_r(v,u)=2r-u$. Then, it follows that
    \begin{align*}
    0&< 7r-2=2r-((1-r)-t)\\
    &\leq 2r-(v-t)\\
    &\leq 2r-u=E_r(u,v)-E_r(v,u)\\ 
    &\leq 4r-1=t.
    \end{align*}

    This completes the proof that $\mu_T(E_r)\le t$, and therefore it follows that $\mu(E_r)=t=4r-1$. Finally, point $\PP$ on line segment $\D\E$ has coordinates $\PP(r, 4r-1)$ for $r\in [\tfrac{3}{10},\tfrac13]$, and so $\PP=(\mu_R(E_r),\mu_T(E_r))$. The line segment $\D\E$ is realized.
\end{proof}

\begin{proposition} \label{pr:EF}
    Let $\PP$ be a point on line segment $\E\F$, where $\E(\tfrac13,\tfrac13), \F(0,\tfrac14)$. Then there exists a copula $C \in \CC$, such that $(\mu_R(C),\mu_T(C)) = \PP$.
\end{proposition}

\begin{figure}[h]
    \centering
    \includegraphics[width=0.6\linewidth]{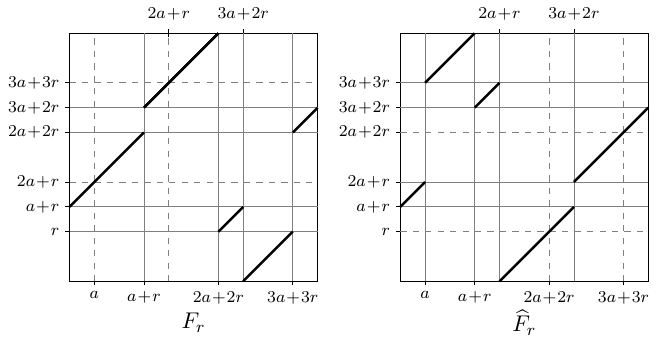}
    \caption{Mass distribution of copulas $F_r$ and $\widehat{F}_r$ for $r=\tfrac15$, hence $a=\tfrac{1}{10}$.}
    \label{fig:10}
\end{figure}

\begin{proof}
Let $r \in [0, \tfrac13]$ and denote by $a = \tfrac{1+r}{4}-r = \tfrac{1-3r}{4}$. Note that $4a + 3r = 1$. Let $F_r$ be shuffle of $M$
\begin{align*}
    F_r &= M(5, (\tfrac{1+r}{4},\tfrac{1+r}{2},\tfrac{3-r}{4},\tfrac{3+3r}{4}), (3,5,2,1,4), (1,1,1,1,1)) \\
    &= M(7, (a,a+r,2a+r,2a+2r,3a+2r,3a+3r), (3,4,6,7,2,1,5), (1,1,1,1,1,1,1)).
\end{align*}
Then 
$$
F_r^T = M(7, (r,a+r,2a+r,2a+2r,3a+2r,4a+2r), (6,5,1,2,7,3,4), (1,1,1,1,1,1,1))
$$
and
$$
\widehat{F}_r = M(7, (a,a+r,2a+r,2a+2r,3a+2r,3a+3r), (3,7,6,1,2,4,5), (1,1,1,1,1,1,1)).
$$
The mass distributions of copulas $F_r$ and $\widehat{F}_r$ are depicted in Figure~\ref{fig:10}.  

For any $r \in [0, \tfrac13]$ we have
$$ F_r(\tfrac{1+r}{4}, \tfrac{1+r}{2}) - \widehat{F}_r(\tfrac{1+r}{4}, \tfrac{1+r}{2}) = F_r(a+r, 2a+2r) - \widehat{F}_r(a+r, 2a+2r)= a+r - a = r,$$
so, $\mu_R(F_r) \ge r$. In order to prove that $\mu_R(F_r) \le r$ we need to show that $|F_r(u,v) - \widehat{F}_r(u,v)| \le r$ for any $(u,v) \in \II^2$. We consider several cases, depending on the value of $u$.

\textbf{Case I.} Suppose first that $u \le a$. In this case $F_r(u,v) = \widehat{F}_r(u,v)$, so $|F_r(u,v) - \widehat{F}_r(u,v)| = 0 \le r$.

\textbf{Case II.} Suppose next that $a \le u \le a+r$. First note that 
\begin{align*}
F_r(a+r, v) &= \begin{cases}
    0; &\text{if } 0 \le v \le a+r, \\
    v-a-r; &\text{if } a+r \le v \le 2a+2r, \\
    a+r; &\text{if } 2a+2r \le v \le 1,
\end{cases} \quad \text{and} \\
\widehat{F}_r(a+r, v) &= \begin{cases}
    0; &\text{if } 0 \le v \le a+r, \\
    v-a-r; &\text{if } a+r \le v \le 2a+r, \\
    a; &\text{if } 2a+r \le v \le 4a+2r, \\
    v-3a-2r; &\text{if } 4a+2r \le v \le 1,
\end{cases}    
\end{align*}
so that 
$$F_r(a+r, v) - \widehat{F}_r(a+r, v) = \begin{cases}
    0; &\text{if } 0 \le v \le 2a+r, \\
    v-2a-r; &\text{if } 2a+r \le v \le 2a+2r, \\
    r; &\text{if } 2a+2r \le v \le 4a+2r, \\
    1-v; &\text{if } 4a+2r \le v \le 1.
\end{cases}
$$
For any $u \in [a, a+r]$ we have
$$0 \le F_r(u, v) - \widehat{F}_r(u, v) \le F_r(a+r, v) - \widehat{F}_r(a+r, v) \le r,$$
so  $|F_r(u,v) - \widehat{F}_r(u,v)| \le r$.

\textbf{Case III.} Assume now that $a+r \le u \le 2a+r$. Note that copulas $F_r$ and $\widehat{F}_r$ have the same mass distribution on the rectangle $[a+r,2a+r]\times[0,1]$, so 
$$F_r(u, v) - \widehat{F}_r(u, v) = F_r(a+r, v) - \widehat{F}_r(a+r, v) \in [0,r]$$
for any $u \in [a+r, 2a+r]$.

\textbf{Case IV.} Assume that $2a+r \le u \le 2a+2r$. In this case we have
\begin{align*}
F_r(u, v) &= \begin{cases}
    0; &\text{if } 0 \le v \le a+r, \\
    v-a-r; &\text{if } a+r \le v \le 2a+2r, \\
    a+r; &\text{if } 2a+2r \le v \le 3a+2r, \\
    v-2a-r; &\text{if } 3a+2r \le v \le 4a+2r, \\
    \min\{u,v-2a-r\}; &\text{if } 4a+2r \le v \le 1, 
\end{cases} \quad \text{and} \\
\widehat{F}_r(u, v) &= \begin{cases}
    \min\{u-2a-r,v\}; &\text{if } 0 \le v \le r, \\
    u-2a-r; &\text{if } r \le v \le a+r, \\
    u+v-3a-2r; &\text{if } a+r \le v \le 2a+r, \\
    u-a-r; &\text{if } 2a+r \le v \le 3a+2r, \\
    u+v-1; &\text{if } 3a+2r \le v \le 1, 
\end{cases}    
\end{align*}
so that 
$$F_r(u, v) - \widehat{F}_r(u, v) = \begin{cases}
    \max\{2a+r-u,-v\}; &\text{if } 0 \le v \le r, \\
    2a+r-u; &\text{if } r \le v \le 2a+r, \\
    v-u; &\text{if } 2a+r \le v \le 2a+2r, \\
    2a+2r-u; &\text{if } 2a+2r \le v \le 4a+2r, \\
    \min\{2a+2r-u,1-v\}; &\text{if } 4a+2r \le v \le 1. 
\end{cases}
$$
It follows that
$$F_r(u, v) - \widehat{F}_r(u, v) \in [-r,r]$$
for any $u \in [2a+r, 2a+2r]$.

\textbf{Case V.} Assume finally that $u \ge 2a+2r$. Then $u \ge \tfrac{1+r}{2} \ge \tfrac12$, so $1-u \le \tfrac 12$. By previous cases we have 
$$|F_r(u,v) - \widehat{F}_r(u,v)| = |F_r(1-u,1-v) - \widehat{F}_r(1-u,1-v)| \le r.$$ 

We have thus proved that $\mu_R(F_r) = r$. 
By Proposition \ref{pr:zgornja1} it follows that $\mu_T(F_r) \le \tfrac{1+r}{4}$. Furthermore, 
$$ F_r(\tfrac{1+r}{4}, \tfrac{1+r}{2}) - F_r(\tfrac{1+r}{2}, \tfrac{1+r}{4})= \tfrac{1+r}{4} - 0 = \tfrac{1+r}{4},$$
so $\mu_T(F_r) = \tfrac{1+r}{4}$.
Finally, point $\PP$ on line segment $\E\F$ has coordinates $\PP(r, \tfrac{1+r}{4})$, where $r \in [0, \tfrac13]$, 
so $\PP = (\mu_R(F_r),\mu_T(F_r))$.
\end{proof}

Next proposition shows that the set $\Omega_{R,T}$ is radially convex (with respect to the origin $\A(0,0)$).

\begin{proposition} \label{pr:radial}
    Let $C \in \CC$ be a copula such that $(\mu_R(C),\mu_T(C)) = \PP$. For any $\lambda \in \II$ there exists copula $C_\lambda \in \CC$, such that $(\mu_R(C_\lambda),\mu_T(C_\lambda)) = \lambda \PP$.
\end{proposition}

\begin{proof}
   Let $C_\lambda = \lambda C + (1-\lambda)M$. Since $\C_\lambda$ is a convex combination of two copulas, it is a copula. Note that $C_\lambda^T = \lambda C^T + (1-\lambda)M$ and $\widehat{C}_\lambda = \lambda \widehat{C} + (1-\lambda)M$, since $M$ is exchangeable and radially symmetric. Now, for any $(u,v) \in \II^2$ we have
   \begin{align*}
       |C_\lambda(u,v)-C_\lambda(v,u)| &= |\lambda C(u,v) + (1-\lambda)M(u,v) - (\lambda C(v,u) + (1-\lambda)M(v,u))| \\
       &= \lambda|C(u,v)-C(v,u)|,
   \end{align*}
   so
   \begin{align*}
       \mu_T(C_\lambda) &= \max\big\{|C_\lambda(u,v)-C_\lambda(v,u)| \colon u,v\in \II\big\} \\
       &= \lambda \max\big\{|C(u,v)-C(v,u)| \colon u,v\in \II\big\} = \lambda\mu_T(C).
   \end{align*}
   Similarly, 
   $|C_\lambda(u,v)-\widehat{C}_\lambda(u,v)| = \lambda|C(u,v)-\widehat{C}(u,v)|$,
   so $\mu_R(C_\lambda) = \lambda\mu_R(C)$. It follows that $(\mu_R(C_\lambda),\mu_T(C_\lambda)) = \lambda \PP$, which finishes the proof.
\end{proof}

We now collect our findings in the main theorem of this section. 

\begin{theorem} \label{th:realization}
       Let $\PP$ be a point in hexagon $\A\B\C\D\E\F$. Then there exists a copula $C \in \CC$, such that $(\mu_R(C),\mu_T(C)) = \PP$.
\end{theorem}

\begin{proof}
   Note that hexagon $\A\B\C\D\E\F$ is a union of four triangles
   $$\A\B\C\D\E\F = \A\B\C \cup \A\C\D \cup \A\D\E \cup \A\E\F,$$
   so, for any point $\PP \in \A\B\C\D\E\F$ there exist a point $\T$ on the union of the edges $\B\C \cup \C\D \cup \D\E \cup \E\F$ and $\lambda \in \II$ such that $\PP = \lambda \T$. By Propositions \ref{pr:BC}, \ref{pr:CD}, \ref{pr:DE}, and \ref{pr:EF} there exists a copula $C \in \CC$, such that $(\mu_R(C),\mu_T(C)) = \T$. By Proposition \ref{pr:radial} there exists a copula $C_\lambda \in \CC$, such that $(\mu_R(C_\lambda),\mu_T(C_\lambda)) = \lambda \T = \PP$, which finishes the proof. 
\end{proof}

\section{Concluding remarks} \label{sec:concluding}

We determine the range of the non-exchangeability of a copula if its radial asymmetry is given and the other way around. The resulting polygon of all possible pairs $(\mu_R(C),\mu_T(C))$ is a non-convex hexagon. The area of this hexagon is equal to $\tfrac{17}{180}$, which represents $85\%$ of the area of the square $[0,\tfrac13]^2$. Our results demonstrate that both types of the asymmetry are not closely related. This is intuitively clear since they measure two very different types of asymmetry.

\section*{Acknowledgments}

The authors acknowledge financial support from the ARIS (Slovenian Research and Innovation Agency), research core funding No. P1-0222 (Damjana Kokol Bukov\v{s}ek, Bla\v{z} Moj\v{s}kerc, and Toma\v{z} Ko\v{s}ir after January 1, 2026), No. P1-0448 (Matja\v{z} Omladi\v{c}, and Toma\v z Ko\v sir until December 31, 2025), and projects J1-70034 (Damjana Kokol Bukov\v{s}ek, Bla\v{z} Moj\v{s}kerc) and J1-50002 (Toma\v{z} Ko\v{s}ir).

\bibliographystyle{plain}
\bibliography{biblio}

\end{document}